\documentclass[12pt,a4paper]{article}
\usepackage[T2A]{fontenc}
\usepackage[cp1251]{inputenc}
\usepackage[english]{babel}

\usepackage{amssymb}
\usepackage{amsmath}
\usepackage{amsthm}
\usepackage{amsfonts}
\usepackage{amscd}

\usepackage{amsmath}
\usepackage{amsfonts}
\usepackage{amssymb}

\usepackage{hyperref}

\newcommand{\seqnum}[1]{\href{http://oeis.org/#1}{\underline{#1}}}

\theoremstyle{plain}
\newtheorem{theorem}{Theorem}
\newtheorem{corollary}[theorem]{Corollary}

\newtheorem{proposition}[theorem]{Proposition}

\theoremstyle{definition}

\begin{document}

\newcommand{\stirlingone}[2]{\left[{#1 \atop #2}\right]}
\newcommand{\stirlingtwo}[2]{\left\{{#1 \atop #2}\right\}}
\newcommand{\eulereantwo}[2]{\left < \!\! \left < {#1 \atop #2} \right > \!\!\right >}
\newcommand{\Eulerian}[2]{\genfrac{<}{>}{0pt}{}{#1}{#2}}

\begin{center}
\vskip 1cm{\LARGE\bf 
Identities and Generating Functions of Products of Generalized Fibonacci numbers, Catalan and Harmonic Numbers
}
\vskip 1cm
\large
Vladimir V. Kruchinin\\
Tomsk State University of Control Systems and Radioelectronics\\
Tomsk \\
Russian Federation\\ 
\href{mailto:kru@ie.tusur.ru}{\tt kru@ie.tusur.ru}

\vskip 1cm
Maria Y. Perminova\\
Tomsk State University of Control Systems and Radioelectronics\\
Tomsk \\
Russian Federation\\ 
\href{mailto:pmy@fdo.tusur.ru}{\tt pmy@fdo.tusur.ru}
\end{center}

\vskip .2 in


\begin{abstract}
We considered the properties of generalized Fibonacci and Lucas numbers class. The analogues of well-known Fibonacci identities for generalized numbers are obtained. We gained a new identity of product convolution of generalized Fibonacci and Lucas numbers. We wrote down generating functions of generalized Fibonacci and Lucas numbers products, their multisections, harmonic numbers and Catalan numbers 
\end{abstract}

\section{Introduction}

The Fibonacci, Lucas, and Catalan numbers are among the most studied objects of combinatorics. There is a huge variety of publications and monographs \cite{Koshy, Andrica, Honsberger, Stanley}. Studies on the products of such numbers, e.g., Fibonacci and Catalan numbers product have appeared relatively recently \cite{Barry}. There are many different generalizations of Fibonacci and Lucas numbers \cite{Bicknell, Dujella, Horadam, Falcon, Edson, Simsek}. Kalman and R. Mena proposed a generalization of Fibonacci numbers of the form 

\begin{equation}\label{LFiba}
F(n)=\left\{
\begin{array}{ll}
0, & n\le 0,\\
1, & n=1,\\
a\,F(n-1)+b\,F(n-2), & n>1.\\
\end{array}
\right. 
\end{equation}
where  $a, b\in \mathbb R$.  
The generating function for generalized Fibonacci numbers is
$$
F(x)=\frac{x}{1-a\,x-b\,x^2}.
$$
They also introduced generalized Lucas numbers
$$
L(n)=F(n+1)+b\,F(n-1).
$$
The generating function for generalized Lucas numbers is
$$
L(x)=\frac{2-a\,x}{1-a\,x-b\,x^2}.
$$
They presented the identity accomplished for generalized Fibonacci numbers
\begin{equation}\label{Equa1}
F(n+m)=b\,F(n-1)F(m)+F(n)\,F(m+1).
\end{equation}

We note if $a, b\in \mathbb Z$ are integers. Then $F(n)$ and $L(n)$ are  Lucas sequences \cite{Lehmer}. When setting the parameters values $a$ and $b$, we obtain a particular sequence. Then fixing the parameters $a$ and $b$, we explicitly indicate their value, e.g.,  $a=1$, $b=1$, $F_{1,1}(n)$ are Fibonacci numbers,  $L_{1,1}(n)$ are Lucas numbers, $a=2$, $b=1$ $F_{2,1}(n)$ are Pell numbers and $L_{2,1}(n)$ are Pell-Lucas numbers.
Table \ref{TablOeis} presents the sequences $F(n)$ and $L(n)$ for different values $a$ and $b$, written in OEIS \cite{OEIS}.
\begin{table}\label{TablOeis}
\begin{center}
\begin{tabular}{llcc}\hline
$a$ & $b$ & $F_{a,b}(n)$ & $L_{a,b}(n)$ \\ \hline
1 &1& \seqnum{A000045} & \seqnum{A000032}\\
2 & 1& \seqnum{A000129} & \seqnum{A002203}\\
3 &1& \seqnum{A006190} & \seqnum{A006497}\\
4 &1& \seqnum{A001076} & \seqnum{A014448}\\
5 &1& \seqnum{A052918} & \seqnum{A087130}\\
6 &1& \seqnum{A005668} & \seqnum{A085447}\\
7 &1& \seqnum{A054413} & \seqnum{A086902}\\
8 &1& \seqnum{A041025} & \seqnum{A086902}\\
9 &1& \seqnum{A099371} & \seqnum{A087798}\\
1 &2& \seqnum{A001045} & \seqnum{A014551}\\
1 &3&\seqnum{A006130} & \seqnum{A075118}\\
2 &2& \seqnum{A002605} & \seqnum{A080040}\\
2 &3& \seqnum{A015518} & \seqnum{A102345}\\
3 &2& \seqnum{A007482} & \seqnum{A206776}\\
3 &3& \seqnum{A030195} & \seqnum{A172012}\\
\end{tabular}\caption{Generalized Fibonacci and Lucas numbers written in OEIS}
\end{center}
\end{table}

We write down a number of identities to obtain generating functions in the following sections.
\begin{proposition}
We accomplish the identity for generalized Fibonacci and Lucas numbers
\begin{equation}\label{Prop1}
2F(n+m)=F(m)L(n)+L(m)F(n).
\end{equation}
\end{proposition}
\begin{proof}
We use the identity (\ref{Equa1})
$$
F(n+m)=b\,F(n-1)F(m)+F(n)\,F(m+1).
$$
and
$$
F(n+m)=b\,F(n)F(m-1)+F(n+1)\,F(m).
$$
We add
$$
2F(n+m)=b\,F(n)F(m-1)+F(n+1)\,F(m)+b\,F(n-1)F(m)+F(n)\,F(m+1),
$$
$$
2F(n+m)=[b\,F(n)F(m-1)+F(n)\,F(m+1])+[F(n+1)\,F(m)+b\,F(n-1)F(m)].
$$
We group
$$
2F(n+m)=F(n)[b\,F(m-1)+F(m+1])+\,F(m)[F(n+1)+b\,F(n-1)].
$$
We make a replacement
$$
L(n)=F(n+1)+b\,F(n-1).
$$
We obtain the sought-for formula
$$
2F(n+m)=F(m)L(n)+L(m)F(n).
$$
\end{proof}

The consequence of this statement at $n=m$ is the identity
\begin{equation}\label{LF2}
F(2\,n)=F(n)L(n).
\end{equation}
We accomplish Johnson's identity for generalized Fibonacci numbers (2003) \cite{Johnson}.
\begin{proposition}
We accomplish the identity for all $p,q,c,d\in \mathbb N$, $r\in \mathbb Z$, at $p+q=c+d$
\begin{equation}\label{Prop2}
F(p)F(q)-F(c)F(d)=(-b)^r\left(F(p-r)F(q-r)-F(c-r)F(d-r)\right). 
\end{equation}
\end{proposition}
\begin{proof}

We use the identity (\ref{Equa1})
$$
F(p+q-1)=b\,F(p-1)F(q-1)+F(p)\,F(q),
$$
$$
F(c+d-1)=b\,F(c-1)F(d-1)+F(c)\,F(d).
$$
Then
$$
b\,F(p-1)F(q-1)+F(p)\,F(q)=b\,F(c-1)F(d-1)+F(c)\,F(d).
$$
We transform to
$$
F(p)\,F(q)-F(c)\,F(d)=-b\left(F(p-1)F(q-1)-F(c-1)F(d-1)\right).
$$
Applying the identity $r$ times for the right side, we obtain the sought-for identity 
$$
F(p)F(q)-F(c)F(d)=(-b)^r\left(F(p-r)F(q-r)-F(c-r)F(d-r)\right). 
$$
\end{proof}

We accomplish the multiple-angle recurrence identity for generalized Fibonacci numbers. 
\begin{proposition}
We accomplish the identity for all $a,b\in \mathbb R$
\begin{equation}\label{Prop3}
F(m\,n+j)=L(m)F(m(n-1)+j)-(-b)^m\,F(m(n-2)+j),
\end{equation} 
where $L(m)$ are generalized Fibonacci numbers.
\end{proposition}
\begin{proof}

We write down generalized Johnson's identity (\ref{Prop2}) $p=m(n-1)+j$, $q=m+1$, $c=m(n-1)+j+1$, $d=m$, $r=m$. 
We obtain
$$
F(m(n-2)+m+j)F(m+1)-F(m(n-2)+m+j+1)F(m)=(-b)^mF(m(n-2)+j),
$$
$$
F(m(n-2)+m+j+1)F(m)=F(m(n-2)+m+j)F(m+1)-(-b)^mF(m(n-2)+j).
$$
We add to the left and to the right parts
$$
bF(m(n-1)+j)F(m-1).
$$
We obtain
\begin{eqnarray}\nonumber
&bF(m(n-1)+j)F(m-1)+F(m(n-1)+j+1)F(m)=\\ 
\nonumber
&F(m(n-1)+j)F(m+1)+b\,F(m(n-1)+j)F(m-1)-(-b)^mF(m(n-2)+j).
\end{eqnarray}
Whence the left part of the identity is equal to $F(mn+j)$. In the right part the first two summands are equal  to $L(m)F(m(n-1)+j)$. Then we obtain the sought-for recurrence formula
$$
F(mn+j)=L(m)F(m(n-1)+j)-(-b)^mF(m(n-2)+j).
$$
\end{proof}

\begin{proposition}
We accomplish the identity for generalized Fibonacci numbers
\begin{equation}\label{Prop4}
F(n+j)F(m)-F(j)F(n+m)=(-1)^j\,b^jF(n)F(m-j).
\end{equation}
\end{proposition}
\begin{proof}

Using Johnson identity (\ref{Prop2}) at $p=n+j$, $q=m$, $c=j$, $d=n+m$, $r=j$ we obtain the sought-for identity
$$
F(n+j)F(m)-F(j)F(n+m)=(-1)^jb^jF(n)F(m-j).
$$
\end{proof}

\section{The identity for generalized Fibonacci and Lucas numbers}
We consider the property of numbers sequences $F(n)$ and $L(n)$.
Let we have some convolution of the form 
$$
p(n)=\sum_{i=0}^n r(i)\,r(n-i),
$$
where $r(n)\in R$.

We prove the following theorem. 
\begin{theorem}
We accomplish the identity for numbers sequences $F(n)$, $L(n)$  and $p(n)$, $k, m \in \mathbb N$ 
\begin{equation}\label{LIden2}
p(n)\,{F}(k\,n+2\,m)=\sum_{i=0}^n r(i){L}(k\,i+m)r(n-i){F}(k(n-i)+m).
\end{equation}
\end{theorem}
\begin{proof}

We obtain the identity (\ref{LIden2}) at $n=0$
$$
p(0)\,{F}(0+2\,m)=p(0)\,L(m)F(m).
$$
Whence we obtain the equality based on the identity (\ref{LF2}).
We now consider the set of pairs of terms in the right part (\ref{LIden2}) such that
$$
r(i)r(n-i)=r(n-i)r(i).
$$
Then for the sum
$$
r(i){L}(k\,i+m)r(n-i)F(k(n-i)+m)+r(n-i){L}(k(n-i)+m)r(i){F}(ki+m)
$$
we accomplish
$$
r(i)r(n-i)\left[{L}(ki+m){F}(k(n-i)+m)+{\it L}(k(n-i)+m){F}(k\,i+m)\right].
$$
We write down on the basis of the identity (\ref{Prop1}) 
$$
{L}(ki+m){F}(k(n-i)+m)+{L}(k(n-i)+m){F}(ki+m)=2F(ki+m+k(n-i)+m)=2\,{F}(kn+2m).
$$
The set of pairs of terms in the right part (\ref{LIden2}) at odd $n$ is even and the obtained identity is accomplished for all pairs. The set of pairs of terms in the right part (\ref{LIden2}) at even $n$ is odd. We consider $n=j+j$. Then
$$
r(j){L}(kj+m)r(j){F}(kj+m)=r(j)r(j){L}(kj+m){F}(kj+m).
$$
We obtain on the basis of the identity (\ref{LF2})
$$
{L}(kj+m){F}(kj+m)=F(2kj+2m)=F(kn+2\,m).
$$
Then we can write down the right part of the identity (\ref{LIden2}) as
$$
\sum_{i=0}^n r(i){\it L}(ki+m)r(n-i){F}(k(n-i)+m)={F}(kn+2m)\,\sum_{i=0}^n r(i)r(n-i).
$$
Knowing that
$$
p(n)=\sum_{i=0}^n r(i)r(n-i)
$$
whence we obtain
$$
p(n)\,F(kn+2\,m)=\sum_{i=0}^n r(i){\it L}(ki+m)r(n-i){\it F}(k(n-i)+m).
$$
We complete the proof.
\end{proof}

We consider the examples of using the identity. Since $r(n)=1$, and then the identity (\ref{LIden2}) is
$$
{(n+1)}F(kn+2m)=\sum_{i=0}^n {L(ki+m)}{F(k(n-i)+m)}.
$$
We obtain at $m=n$
$$
{(n+1)}F((k+1)\,n)=\sum_{i=0}^n {L(ki+n)}{F(k(n-i)+n)}.
$$
We now consider the case $r(n)=F(n)$. We substitute into the identity (\ref{LIden2})
$$
p(n)\,{F}(n+2\,m)=\sum_{i=0}^n F(i){L}(i+m)F(n-i){F}(n-i+m). 
$$
where
$$
p(n)=\sum_{i=0}^n F(i)F(n-i).
$$
Hoggart obtained the following explicit expression for the convolution of Fibonacci numbers \cite{HB}
$$
\sum_{i=0}^n F_{1,1}(i)F_{1,1}(n-i)=\frac{1}{5}\left((n-1)\,F_{1,1}(n)+(n+1)F_{1,1}(n-1)\right).
$$
Using his method we can obtain the convolution formula for generalized Fibonacci numbers
\begin{equation}\label{LComvo}
\sum_{i=0}^n F(i)F(n-i)=\frac{(n-1)F(n+1)+b\,(n+1)F(n-1)}{a^2+4\,b}.
\end{equation}
Whence at $m=0$, $k=1$, $a=1$, $b=1$ we obtain the identity
$$
\sum_{i=0}^nF_{1,1}(i)^2\,L_{1,1}(n-i)F_{1,1}(n-i)=F_{1,1}(n)\frac{(n-1)F_{1,1}(n+1)+(n+1)F_{1,1}(n-1)}{5}.
$$
For Pell numbers we have $m=0$, $k=1$, $a=2$, $b=1$
$$
\sum_{i=0}^nF_{2,1}(i)^2\,L_{2,1}(n-i)F_{2,1}(n-i)=F_{2,1}(n)\frac{(n-1)F_{2,1}(n+1)+(n+1)F_{2,1}(n-1)}{8}.
$$
For Jacobsthal numbers we have $m=0$, $k=1$, $a=1$, $b=2$
$$
\sum_{i=0}^n F_{1,2}(i)^2\,L_{1,2}(n-i)F_{1,2}(n-i)=F_{1,2}(n)\frac{(n-1)F_{1,2}(n+1)+(n+1)F_{1,2}(n-1)}{9}.
$$
We can replace the property $r(i)r(n-i)=r(n-i)r(i)$ by the property $T(n,k)=T(n,n-k)$ for a triangle $T(n,k)$, the triangle is symmetric in relation to main diagonal. Then the identity (\ref{LIden2}) is
$$
F(kn+2\,m)\sum_{k=0}^n T(n,k)=\sum_{k=0}^n T(n,k)F(k\,i+m)L(k\,(n-i)+m).
$$
We write down the identity for Pascal's triangle
$$
2^n\,F(kn+2\,m)=\sum_{k=0}^n \binom{n}{k}F(k\,i+m)L(k\,(n-i)+m).
$$
Similarly,
$$
\binom{2\,n}{n}\,F(kn+2\,m)=\sum_{k=0}^n \binom{n}{k}^2\,F(k\,i+m)L(k\,(n-i)+m).
$$
The identity for Euler numbers of the first kind $E_1(n,k)$
$$
(n+1)!\,F(kn+2\,m)=\sum_{k=0}^n E_1(n,k)F(k\,i+m)L(k\,(n-i)+m).
$$
Leibniz's harmonic triangle, Losanitsch's triangle, A027907 triangle of trinomial coefficients and others have such properties \cite{OEIS}.
The obtained identity (\ref{LIden2}) allows us to write down the functional equation of the form
$$
A(x)=B(x)\,C(x),
$$
where 
$$
A(x)=\sum_{n=0}^{\infty} p(n)\,F(kn+2\,m)x^n,
$$
$$
B(x)=\sum_{n=0}^{\infty} r(n)\,F(kn+m)x^n,
$$
$$
C(x)=\sum_{n=0}^{\infty} r(n)\,L(kn+m)x^n.
$$
Now knowing two out of the three functions we can find the third function. We consider obtaining generating functions based on this equation.

\section{Generating functions for the product of generalized Fibonacci and harmonic numbers}

We consider the case $r(n)=\frac{1}{n}$, $m=0$, $k=1$. Then the identity (\ref{LIden2}) is
$$
p(n)F_1(n)=\sum_{i=}^{n-1} \frac{L(i)}{(i)}\frac{F(n-i)}{(n-i)},
$$
where
$F(n)$ are generalized Fibonacci numbers, $L(n)$ are adjoined numbers
$$
p(n)=\sum_{i=1}^{n-1} \frac{1}{i(n-i)}=\frac{2}{n}{{{H}\left(n-1\right)}},
$$
where $H(n)$ are harmonic numbers.
We find the generating function for $p(n)F(n)$. We note that
$$
p(n)F(n)=[x^n]U_1(x)=[x^n]U_2(x)\,U_3(x),
$$
where 
$$
U_2(x)=\sum_{n=1}^{\infty}\frac{F(n)}{n}x^n,
$$
$$
U_3(x)=\sum_{n=1}^{\infty}\frac{L(n)}{n}x^n.
$$
Then 
$$U_2(x,a,b)=\int\frac{1}{(1-a\,x-b\,x^2)}dx=\frac{1}{{\sqrt{4\,b+a^2}}}{{\log \left({{-\sqrt{4\,b+a^2}\,x+a\,x-2}\over{\sqrt{4\,b+a^2}\,x+
 a\,x-2}}\right)}},$$
$$U_3(x,a,b)=\int \frac{2-a\,x}{(1-a\,x-b\,x^2)}dx= -\log \left(1-a\,x-b\,x^2\right).
$$
Then the sought-for generating function is
$$
U_1(x,a,b)=\frac{1}{{\sqrt{a^2+4\,b}}}{{\log\left({{2-\sqrt{a^2+4b}\,x-a\,x}\over{2+\sqrt{a^2+4b}\,x-a\,x}}\right)}}\log \left(1-a\,x-b\,x^2\right).
$$
And
$$
\frac{2}{n}{{{H}(n-1)\,{F}(n)}}=[x^n]U_1(x,a,b).
$$
Hence we can write down 
$$
H(n-1)F(n)=[x^n]{{{{x\,d\,U_1(x,a,b)}\over{2\,d\,x}}}}.
$$
Then the generating function for the product $H(n)F(n+1)$ is
$$U_4(x,a,b)={{\left((a+2\,b\,x)\,\log \left({{2-\sqrt{a^2+4b}\,x-a\,x}\over{2+\sqrt{a^2+4b}\,x-a\,x}}\right)+\sqrt{a^2+4\,b}\,
 \log \left(1-a\,x-b\,x^2\right)\right)}\over{2\,\sqrt{a^2+4\,b}\,
 \left(1-a\,x-b\,x^2\right)}}.
$$
We now find the generating function $U_5(x,a,b)$ for $H(n)F(n)$. We write down 
$$
H(n)F(n)=H(n-1)F(n)+\frac{1}{n}F(n).
$$
Whence 
$$
H(n)F(n)=[x^n]{{{{x\,d\,U_1(x,a,b)}\over{2\,d\,x}}}}+U_2(x,a,b).
$$
Then the generating function for the product $H(n)F(n)$ is
$$U_5(x,a,b)=-{{{{\frac{(2-a\,x)}{{\sqrt{4\,b+a^2}}}\,\log \left({{2-\sqrt{a^2+4b}\,x-a\,x}\over{2+\sqrt{a^2+4b}\,x-a\,x}}\right)}}+x
 \,\log \left(1-a\,x-b\,x^2\right)}\over{2\,\left(1-a\,x-b\,x^2\right)}}.
$$
For example, we obtain at $a=1$ and $b=1$ 
$$H(n)F_{1,1}(n)=[x^n]{{5\,x\,\log \left(1-x-x^2\right)+\sqrt{5}(2-x)\,\log \left(-{{\left(\sqrt{5}+1\right)\,x-2}\over{\left(
 \sqrt{5}-1\right)\,x+2}}\right)}\over{10(x^2+x-1)}}.
$$
The given sequence is written in A372199. The presented series converges at $x=\frac{1}{2}$. 
Hence it follows that
$$\sum_{n=1}^{\infty }{{{{\it H}\left(n\right)\,{F_1}\left(n
 \right)}\over{2^{n}}}}={{\frac{3}{\sqrt{5}}\log \left({{3+\sqrt{5}}\over{3-\sqrt{5}}}\right)+
\log 4}}.
$$
We obtain the generating function for the product of harmonic numbers and Pell numbers at $a=2$ and $b=1$
 $$H(n)F_{2,1}(n)=[x^n]{{2\,x\,\log(1-2x-x^2)+\sqrt{2}\left(1-x
 \right)\,\log \left({{1-\left(\sqrt{2}+1\right)\,x}\over{1+\left( \sqrt{2}-1\right)\,x}}\right)}\over{4\,\left(x^2+2\,x-1\right)}}.$$
The given sequence is written in A372210. The presented series converges at $x=\frac{1}{3}$.
Hence it follows that
$$\sum_{n=1}^{\infty }\frac{1}{3^n}H(n)F_{2,1}(n)=-\frac{3}{4}\left({{\sqrt{2}\,\log \left({{2-\sqrt{2}}\over{2+\sqrt{2}}}\right)+
 \log \left({{2}\over{9}}\right)}}\right).$$
The obtained generating functions $U_4(x,a,b)$ and $U_5(x,a,b)$ allow us to write down the generating function for the product $H(n)F(n+j+1)$, $j\ge 0$. For this purpose we use the identity (\ref{Equa1})
$$
H(n)F(n+j+1)=F(j+1)H(n)F(n+1)+b\,F(j)H(n)F(n).
$$
Whence the generating function has the expression
\begin{equation}
U_6(x,a,b)=F(j+1)U_5(x,a,b)+b\,F(j)U_4(x,a,b).
\end{equation}

\section{Generating functions for the product of generalized Fibonacci and Catalan numbers}
We consider the application of the obtained identity (\ref{LIden2}) at $r(n)=C(n)$. Where $C(n)$ are Catalan numbers
$$
C(n)=\frac{1}{n+1}\binom{2\,n}{n}.
$$
Then the identity is 
\begin{equation}\label{LCatId}
C(n+1)\,{F}(n+2\,m)=\sum_{i=0}^n C(i){L}(i+m)C(n-i){F}(n-i+m).
\end{equation}
We write down the known recurrence relation
$$
C(n+1)=\sum_{i=0}^n C(i)C(n-i), C(0)=1.
$$ 
We write down the result obtained by Barry and Mwafise \cite{Barry}
\begin{equation}\label{FBarry}
C(n)p(n+1)=[x^n]V_1(x,a,b)=[x^n]\frac{1}{x}\sqrt{\frac{1-2ax-\sqrt{1-4ax-16bx^2}}{(2(a^2+4b))}},
\end{equation}
where
$$
p(n)=[x^n]\frac{1}{1-ax-bx^2}.
$$
We write down the generating function
$$
V_2(x,a,b)=\sum_{n=0} C(n)\,{L}(n+1)x^n.
$$
\begin{corollary}
The generating function for the product of generalized Lucas and Catalan numbers of the form
$$
C(n)\,L(n+1)
$$
has the following expression
\begin{equation}\label{LGF1}
V_2(x,a,b)=\frac{1}{x}\left({{1-{{\sqrt{1-2\,a\,x+\sqrt{1-4\,a\,x-16\,b\,x^2}\over{{2}
 }}}}}}\right).
\end{equation}
\end{corollary}
\begin{proof}
We obtain the recurrence equation based on the identity (\ref{LIden2}) at $m=1$
$$
C(n+1)\,{F}(n+2)=\sum_{i=0}^n C(i){\it L}(i+1)C(n-i){F}(n-i+1).
$$
Then we can write down the functional equation linking the generating functions $V_1(x,a,b)$ and $V_2(x,a,b)$
$$
V_1(x,a,b)=1+x\,V_1(x,a,b)\,V_2(x,a,b).
$$
Whence we find the explicit expression for the sought-for generating function
$$
V_2(x,a,b)=\frac{1}{x}\left(1-\frac{1}{V_1(x,a,b)}\right).
$$
\end{proof}

\begin{corollary}

The generating function for the product of generalized Fibonacci and Catalan numbers
$C(n)\,F(n)$
has the following expression
$$V_3(x,a,b)=\frac{1}{2bx}\left({{1-\sqrt{{{2\,b\,\sqrt{-16\,b\,x^2-4\,a\,x+1}+4\,a\,b\,x+2\,b+a^2
 }\over{4\,b+a^2}}}}}\right).
$$
\end{corollary}
 
\begin{proof}
By definition we have
$$
L(n+1)=F(n+2)+b\,F(n)=a\,F(n+1)+2\,b\,F(n).
$$
Then
$$
C(n)L(n+1)=a\,C(n)F(n+1)+2\,b\,F(n)C(n)
$$
$$
V_2(x,a,b)=a\,V_1(x,a,b)+2\,b\,V_3(x,a,b)
$$

$$
V_3(x,a,b)=\frac{1}{2\,b}\left(V_2(x,a,b)-a\,V_1(x,a,b)\right).
$$
Hence knowing the generating functions for the right part terms of the expression we obtain the sought-for expression.
\end{proof}

Then for the product of Fibonacci and Catalan numbers the generating function is
$$
F_1(n)C(n)=[x^n]\frac{1}{2\,x}\left({{1-{{\sqrt{2\,\sqrt{1-4\,x-16\,x^2}+4\,x+3}}\over{
 \sqrt{5}}}}}\right).
$$
We write the generating function in A119694. At $x=\frac{1}{8}$, the series converges. Then
$$
\sum_{n=1}^{\infty} \frac{1}{2^{3n}}C(n)F_1(n)=4\,\left(1-{{3}\over{\sqrt{10}}}\right).
$$
For the product of Pell and Catalan numbers the generating function is 
$$
F_2(n)C(n)=[x^n]\frac{1}{2\,x}\left({{1-{{\sqrt{2\,\sqrt{1-8\,x-16\,x^2}+8\,x+6}}\over{2\sqrt{2}}}}}\right).
$$
We write the sequence in A372216. At $x=\frac{1}{16}$ the series converges.
Then
$$\sum_{n=1}^{\infty} \frac{1}{2^{4n}}C(n)F_2(n)=8-2\,\sqrt{13+\sqrt{7}}.$$

The generating functions $V_1(x,a,b)$ and $V_3(x,a,b)$ allow us to write down the generating function for the product $C(n)F(n+j+1)$, $j\ge 0$. For this purpose we use the identity (\ref{Equa1}).
Whence the generating function has the expression 
\begin{equation}
V_4(x,a,b)=F(j+1)V_1(x,a,b)+b\,F(j)V_3(x,a,b).
\end{equation}

\subsection{Generating functions for multisections product  of generalized Fibonacci and Lucas numbers, Catalan and harmonic numbers } 

The generating functions for multisections of Fibonacci and Lucas numbers are known \cite{Shallit,Gessel}. We find the expression of the generating function for multisections of generalized Fibonacci numbers of the form $F(m\,n+j)$. For this purpose we prove the following proposition
\begin{proposition}
For multisections of generalized Fibonacci numbers $F(m\,n+j)$, $m,n,j\in N$, $a,b\in R$ the generating function is
\begin{equation}\label{MsecFib}
Q(x)={{{F}(j)+(-b)^{j}\,{F}(m-j)\,x
 }\over{1-L(m)\,x+(-b)^{m}\,x^2}},
\end{equation}
where $L(m)$ is generalized  Lucas number with parameters $a,b$.
\end{proposition}
\begin{proof}
We consider the obtained multiple-angle recurrence (\ref{Prop3}) at $j=0$
$$
F(m\,n)=L(m)F(m(n-1))-(b)^m\,F(m(n-2)).
$$
We easily see the generating function is
$$
F(m\,n)=[x^n]\frac{F(m)x}{1-L(m)x+(-b)^m\,x^2}.
$$
We now use the identity (\ref{Prop3}). Then we select $F(n+j)$ in the left part of $F(n+j)$ and replace the argument $n$ by $m\,n$
$$
F(mn+j)=\frac{F(j)}{(F(m)}F(m(n+1))+\frac{(-b)^j\,F(m-j)}{F(m)}F(mn).
$$
Whence
$$
F(mn+j)=[x^n]\frac{F(j)}{(F(m)}\frac{F(m)x}{x(1-L(m)x+(-b)^m\,x^2)}+\frac{(-b)^j\,F(m-j)}{F(m)}\frac{F(m)x}{(1-L(m)x+(-b)^m\,x^2)}.
$$
Then we obtain the sought-for formula after simplification.
\end{proof}

We use the relation between generalized Fibonacci and Lucas numbers. We obtain the generating function for generalized Lucas numbers
\begin{equation}\label{MsecLuc}
Q_L(x)={{{L}(j)-(-b)^{j}\,{L}(m-j)\,x
 }\over{1-L(m)\,x+(-b)^{m}\,x^2}}.
\end{equation}

We note that at $j=0$ multisection of generalized Fibonacci numbers of the form $F(m\,n)$ belongs to generalized Fibonacci numbers with parameters $a=L(m)$, $b=-(-b')^m$, where $a'$ and $b'$ are parameters of the sequence $F(n)$ generating the multisection, $L(m)$ are generalized Lucas numbers with parameters $a'$ and $b'$. Then we can use all identities valid for generalized Fibonacci numbers for multisections of the form $F(m\,n)$. Then we can use the function $U_5(x,a,b)$ to determine the generating functions of multisections product of generalized Fibonacci numbers of the form $F(m\,n)$ and harmonic numbers $H(n)$. We can use the function $V_3(x,a,b)$ to determine the generating functions of multisections product of generalized Fibonacci numbers of the form $F(m\,n)$ and Catalan numbers $H(n)$. We cannot use the proposed approach for multisections of generalized Lucas numbers since $L(0)\ne 0$). 
We consider some examples.\\
1. Let $a=2$ $b=1$ $F_{2,1}(n)$ are Pell numbers, $j=0$ and $m=2$. Then $F_{2,1}(2)=2$, $L_{2,1}(2)=6$, the generating function for bisection of Pell numbers is
$$
F_{2,1}(2n)=[x^n]\frac{2x}{1-6x-x^2}.
$$
Whence
$$
H(n)F_{2,1}(2n)=[x^n]2\,U_5(x,6,-1),
$$
$$
C(n)F_{2,1}(2n)=[x^n]2\,V_3(x,6,-1).
$$
2. Let $a=1$, $b=2$, $F_{1,2}(n)$ are Jacobsthal numbers, $j=0$ and $m=3$. Then $F_{1,2}(3)=3$, $L_{2,1}(3)=7$ the generating function for trisection of Jacobstal numbers is
$$
F_{2,1}(3\,n)=[x^n]\frac{3x}{1-7x-8x^2}.
$$
Whence
$$
H(n)F_{1,2}(3\,n)=[x^n]3\,U_5(x,7,8)
$$
$$
C(n)F_{1,2}(3\,n)=[x^n]3\,V_3(x,7,8)
$$
We consider obtaining generating functions for the product $H(n)F(mn+j)$ and $C(n)F(mn+j)$ at $0\leqslant j<m$. For this purpose we use the identity (\ref{Prop4}) and replace the variable $n$ by $nm$. We obtain
$$
F(m\,n+j)F(m)-F(j)F(m\,n+m)=(-1)^j\,b^jF(m\,n)F(m-j).
$$
Then
$$
F(m\,n+j)=\frac{F(j)}{F(m)}F(m(n+1)+\frac{(-1)^j\,b^j\,F(m-j)}{F(m)}F(m\,n).
$$
We multiply the left and right parts by $H(n)$
$$
H(n)F(m\,n+j)=\frac{F(j)}{F(m)}H(n)F(m(n+1)+\frac{(-1)^j\,b^j\,F(m-j)}{F(m)}H(n)F(m\,n).
$$
We now proceed to the generating functions
$$
H(n)F(m(n+1))=[x^n]F(m)U_4(x,L(m),-(-b)^m),
$$
and
$$
H(n)F(mn)=[x^n]F(m)U_5(x,L(m),-(-b)^m).
$$
After simplification we obtain
$$
H(n)F(m\,n+j)=[x^n]F(j)U_4(x,L(m),-(-b)^m)+(-b)^j\,F(m-j)U_5(x,L(m),-(-b)^m).
$$
We perform similar actions for the product $C(n)F(mn+j)$ and write down the generating function
$$
C(n)F(m\,n+j)=[x^n]F(j)V_1(x,L(m),-(-b)^m)+(-b)^j\,F(m-j)V_3(x,L(m),-(-b)^m).
$$
We use the formula for generalized Lucas numbers
$$
L(m\,n+j)=F(m\,n+j+1)+b\,F(m\,n+j-1)
$$
After simplification we obtain
$$
H(n)L(m\,n+j)=[x^n]L(j)U_4(x,L(m),-(-b)^m)-(-b)^j\,L(m-j)U_5(x,L(m),-(-b)^m)
$$
and
$$
C(n)L(m\,n+j)=[x^n]L(j)V_1(x,L(m),-(-b)^m)-(-b)^j\,L(m-j)V_3(x,L(m),-(-b)^m).
$$
We consider the examples. \\
1. We specify Jacobsthal numbers $a=1$, $b=2$ $j=3$, $m=4$. Then the generating function for the product $H(n)F_{1,2}(4n+1)$ is
$$
H(n)F_{1,2}(4n+3)=[x^n]{{\left(8\,x+7\right)\,\log \left({{x-1}\over{16\,x-1}}\right)+\left(24\,x-9\right)\,\log \left(16\,x^2-17\,x+1\right)}\over{96\,x^
 2-102\,x+6}}.
$$
2. We specify the sequence of numbers \seqnum{A010892}, $m=4$, $j=1$, $a=1$, $b=-1$. Then the product $С(n)F_{1,-1}(4n+1)$ has the generating function 
$$C(n)F_{1,-1}(4n+1)=[x^n]{{\sqrt{\sqrt{16\,x^2+4\,x+1}-2\,x-1}}\over{\sqrt{6}\,x}}.
$$
3. We specify the sequence of Jacobsthal-Lucas numbers \seqnum{A014551}, $m=4$, $j=1$, $a=1$, $b=-1$. Then the product $H(n)L_{1,2}(2n)$ has the generating function

$$H(n)L_{1,2}(2n)=[x^n]{{3\,x\,\log \left({{x-1}\over{4\,x-1}}\right)+\left(5\,x-2\right)
 \,\log \left(4\,x^2-5\,x+1\right)}\over{8\,x^2-10\,x+2}}.
$$
4. We specify the sequence \seqnum{A002249}, $a=1$, $b=-2$, $j=1$, $m=3$. Then the product $C(n)L_{1,-2}(2n)$ has the generating function
$$C(n)L_{1,-2}(3n+1)=[x^n]{{\sqrt{\sqrt{128\,x^2+16\,x+1}-8\,x-1}}\over{\sqrt{8}\,x}}.$$

\section{Conclusion}

1. We wrote down the identities for generalized Fibonacci and Lucas numbers.\\
2. We obtained the new identity for generalized Fibonacci and Lucas numbers
$$
p(n)\,{F}(k\,n+2\,m)=\sum_{i=0}^n r(i){L}(k\,i+m)r(n-i){F}(k(n-i)+m),
$$
where
$$
p(n)=\sum_{i=0}^n r(i)r(n-i).
$$ 
3. We obtained the generating functions for the products of the form
$$
H(n)F(n+1)=[x^n]U_4(x,a,b),
$$  
$$
H(n)F(n)=[x^n]U_5(x,a,b),
$$  
$$
C(n)F(n)=[x^n]V_3(x,a,b).
$$  
4. We obtained the generating functions for the products of the form
 $$H(n)F(n+j+1)=[x^n]F(j+1)U_5(x,a,b)+b\,F(j)U_4(x,a,b),$$
 $$C(n)F(n+j+1)=[x^n]F(j+1)V_1(x,a,b)+b\,F(j)V_3(x,a,b).$$
5. We obtained the generating functions for multisections of generalized Fibonacci and Lucas numbers
$$
Q_F(x)={{{F}(j)+(-b)^{j}\,{F}(m-j)\,x
 }\over{1-L(m)\,x+(-b)^{m}\,x^2}},
$$
$$
Q_L(x)={{{L}(j)-(-b)^{j}\,{L}(m-j)\,x
 }\over{1-L(m)\,x+(-b)^{m}\,x^2}}.
$$
6. We obtained the generating functions for multisections product of generalized Fibonacci and Lucas numbers, Catalan and harmonic numbers
$$
H(n)F(m\,n+j)=[x^n]F(j)U_4(x,L(m),-(-b)^m)+(-b)^j\,F(m-j)U_5(x,L(m),-(-b)^m),$$
$$
C(n)F(m\,n+j)=[x^n]F(j)V_1(x,L(m),-(-b)^m)+(-b)^j\,F(m-j)V_3(x,L(m),-(-b)^m),$$
$$
H(n)L(m\,n+j)=[x^n]L(j)U_4(x,L(m),-(-b)^m)-(-b)^j\,L(m-j)U_5(x,L(m),-(-b)^m),$$
$$
C(n)L(m\,n+j)=[x^n]L(j)V_1(x,L(m),-(-b)^m)-(-b)^j\,L(m-j)V_3(x,L(m),-(-b)^m). 
$$

\end{document}